\documentclass[12pt, psamsfonts, reqno]{amsart}

\usepackage{amssymb,amsfonts,amsmath,graphicx,lineno}
\usepackage{mathrsfs}
\usepackage{anysize}
\usepackage[colorlinks, citecolor=black, urlcolor=black, linkcolor=black]{hyperref}
\usepackage{url}
\usepackage{color}
\usepackage{tikz}
\usetikzlibrary{shapes, arrows}
\usepackage{pgfplots}
\tikzset{>= angle 60}

\usepackage{graphics}
\usepackage{amsmath}
\usepackage{amsthm}
\usepackage{amssymb}
\usepackage{amscd}
\usepackage{amsfonts}
\usepackage{amsbsy}
\usepackage{epsfig}
\usepackage{color}

\newtheorem{theorem}{Theorem}

\newtheorem{definition}{Definition}
\newtheorem{corollary}[theorem]{Corollary}
\newtheorem{lemma}[theorem]{Lemma}

\newtheorem*{conjecture}{Conjecture}

\newcommand{\cqd}{\hfill $\square$}

\def\R{{\mathbb R}}

\def\Z{{\mathbb Z}}

\def\E{{\mathcal E}}

\def\o{\mathrm o}

\begin{document}

\title{Rotation number of interval contracted rotations}

\maketitle

\centerline{ Michel Laurent and  Arnaldo Nogueira}

\footnote{\footnote \rm 2010 {\it Mathematics Subject Classification:}   
11J91,   37E05. }


\marginsize{2.5cm}{2.5cm}{1cm}{2cm}
  \begin{abstract} Let $0<\lambda<1$. We consider the one-parameter family of circle $\lambda$-affine contractions $f_\delta:x \in [0,1) \mapsto \lambda x + \delta \; {\rm mod}\,1 $, where $0 \le \delta <1$. Let $\rho$ be the rotation number of the map $f_\delta$. We will  give some  numerical relations between the values of  $\lambda,\delta$ and $\rho$, essentially using  Hecke-Mahler series and a tree structure.  When both parameters $\lambda$ and $\delta$ are algebraic  numbers, we show that  $\rho$ is a rational number.  Moreover, in the case  $\lambda$ and $\delta$ are rational, we give an explicit upper bound for  the  height  of $\rho$ under some assumptions on $\lambda$.
  \end{abstract}

\section{Introduction}\label{sec:intro}

Let $I=[0,1)$ be the unit interval.  

\begin{definition} 
Let $0< \lambda<1$ and $\delta\in I$. We call the map defined by
$$
f=f_{\lambda,\delta}:x\in I \mapsto \{ \lambda x+\delta\}  ,
$$
where the symbol $\{ .\}$ stands for the fractional part, a
{\it  contracted rotation} of $I$. 
\end{definition}

In particular, if $\lambda+\delta>1$, $f$ is a {\it 2-interval piecewise contraction} on the interval $I$ (see Figure 1).

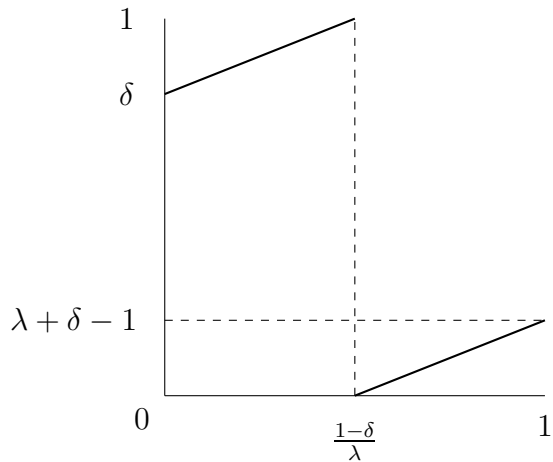
\begin{figure}[ht]
\begin{tikzpicture}
\draw (0,0) -- (5,0);
\draw (0,0) -- (0,5);
\draw[thick] (0,4) -- (2.5,5);
\draw[thick] (2.5,0) -- (5,1);
\node   at (-0.3,-0.3) {$0$};
\node at (2.5,-0.6) {$\frac{1-\delta}{\lambda}$};
\node   at (-0.5,5) {$1$};
\node   at (-0.5,4) {$\delta$};
\draw[dashed] (0,1) -- (5,1);
\node   at (-1.2,1) {$\lambda + \delta -1$};
\node   at (5,-0.4) {$1$};
\draw[dashed] (2.5,0) -- (2.5,5);
\end{tikzpicture}
\caption{ A plot of $f_{\lambda, \delta}: I \to I$, where $\lambda + \delta > 1$}
\end{figure}

 Many authors have studied the dynamics of contracted rotations, as a dynamical system or in applications, amongst others \cite{Br, Bu, BuC, Co, DH, FC, LM, NS}. 
 It is known that every  contracted rotation map $f$ has a rotation number  $\rho = \rho_{\lambda,\delta}$, satisfying $0\le \rho <1$. The classical definition of $\rho$  will be  recalled in Section 5.
   If $ \rho $ takes an irrational value,  then the closure $\overline{C}$ of  the limit set  $C:= \cap_{k\ge 1}f^k(I)$ of $f$  is a Cantor set  and $f$ is topologically conjugated to the rotation map $x \in I \mapsto x + \rho   \mod 1$ on C. When the  rotation number is rational,  the map $f$ has  at most one periodic orbit (see \cite{GT} or Section 5.2 for a simple proof)  and the limit set $C$ equals the periodic orbit if it does exist.

 The goal of this article is to study the value of  the rotation number  $\rho_{\lambda,\delta}$
according to the diophantine nature of the parameters  $\lambda$ and $\delta$. 
Applying a classical transcendence result,
which is stated as  Theorem 5 below,   we   obtain the

\begin{theorem} 
Let $0<\lambda, \delta<1$ be algebraic real numbers. Then,  the rotation number $\rho_{\lambda,\delta}$ is a rational number.
\end{theorem}

In view of Theorem 1, a natural problem that arises  is to  estimate  the height of the rational  rotation number $\rho_{\lambda,\delta}$ in terms of the 
algebraic  numbers $ \lambda $ and $ \delta $.
We  provide a partial solution for this  issue  when $\lambda$ and $\delta$ are rational. Note that Theorem 2  includes the case where  $\lambda$ is the reciprocal of an integer.

\begin{theorem}
 Let $\lambda = a/b$  and  $\delta = r/s$ be rational numbers 
with $0< \lambda, \delta <1$.  Assume that $b> a^\gamma$, where $\gamma = {1+ \sqrt{5}\over 2}$ denotes the golden ratio. Then, the rotation number $\rho_{\lambda,\delta}$ is a rational number $p/q$ where 
$$
0 \le p < q \le  \gamma^{ 2+ { \gamma \log (sb) \over  \log b - \gamma \log a} }.
$$
\end{theorem}

 It should be interesting to extend  the validity of Theorem 2 to a larger class of ratios $a/b$. 
The exponent $\gamma = {1+\sqrt{5}\over 2}$ is best possible with regard to the tools employed in its  proof, as explained  in  Section 3.4  below. However, Theorem 2 should presumably be improved. We propose the following 

\begin{conjecture}
Let $\gamma $ be any real number greater than 1.  Theorem 2 holds true assuming that $b > a^\gamma$, with a possible larger upper bound for $q$ depending only on $\gamma, a,b$ and $s$.
\end{conjecture}

Our proofs of Theorems 1 and 2 are  based on an arithmetical analysis of formulae giving the rotation
number $\rho_{\lambda,\delta}$ in terms of the parameters $\lambda$ and $\delta$. 
As far as we are aware, it is in the works of   E. J. Ding and P. C. Hemmer \cite{DH} and  Y. Bugeaud  \cite{Bu} that appears the first complete description of the relations between the parameters $ \lambda, \delta$ and the rotation number $\rho_{\lambda,\delta}$. 
For $0<\lambda<1$  fixed, these papers deal with the variation of the rotation number in the one-dimensional family of contracted rotations $f_{\lambda,\delta}$ as $\delta$ runs through the interval $[0,1)$. 
We summarize the results that we need in the following 

\begin{theorem} 
Let $0<\lambda<1$ be given. Then the application   $\delta \mapsto \rho_{\lambda, \delta}$
is a continuous non decreasing function sending $I$ onto $I$ and satisfying the following properties: 
 \\
(i) The rotation number  $\rho_{\lambda,\delta}$ vanishes exactly when  $0\le \delta\le 1-\lambda$.
\\
(ii) Let $\displaystyle\frac{p}{q}$ be a positive rational number, where $0<p<q$ are relatively prime integers. Then   $  \rho_{\lambda,\delta}$ takes the value $\displaystyle{p\over q} $ if and only if $\delta$ is located in the interval
$$
\frac{1-\lambda}{1-\lambda^q} c\left(\lambda,\frac{p}{q}\right) \le \delta \le \frac{1-\lambda}{1-\lambda^q} \left( c\left(\lambda,\frac{p}{q}\right) + \lambda^{q-1}-\lambda^q \right),
$$
where
$$
c\left(\lambda,\frac{p}{q}\right) =  1+\sum_{k=1}^{q-2}  \left( \left[ (k+1) \frac{p}{q} \right] -  \left[ k \frac{p}{q} \right] \right)\lambda^{k}
$$
and the above sum equals $0$ when $q=2$.
\\
(iii) For every irrational number $\rho$ with $0 < \rho < 1$, there exists one and only one 
real number $\delta$ such that $0 <\delta < 1$ and   $\rho_{\lambda,\delta}= \rho$
 which is given by the formula
$$
\delta=\delta(\lambda,\rho)=(1-\lambda) \left( 1 +  \sum_{k=1}^{+\infty }\left(  \left[ (k+1) \rho \right] -  \left[ k \rho \right] \right)\lambda^{k} \right).
\eqno{(1)}
$$
\end{theorem}

The proof of Theorem 2 deeply relies on a tree structure, introduced by Y. Bugeaud and J.-P. Conze in \cite{BuC}, which is parallel  to the classical Stern-Brocot tree of rational numbers. It enables us to handle more easily the complicated intervals occuring in Theorem 3 (ii). 
Finally, using the same tools,   we give a simple proof that

\begin{theorem} 
Let $0< \lambda<1$ be given. Then  
$$
\{\delta \in I:\rho_{\lambda,\delta} \mbox{ is an irrational  number}\} 
$$
has zero Hausdorff dimension.
\end{theorem}

 The paper is organized as follows. In Section 2, we prove  Theorem 1 and give another claim (Corollary 7) about the diophantine  nature of the number $\delta(\lambda,\rho)$.  To that purpose, we  use some results coming from Transcendental  Number Theory which are briefly recalled. 
The proof of Theorem 2 is achieved in Section 3. It combines the tree structure already mentioned with an application of Liouville's inequality in our context (Lemma 8). 
In Section 4 we provide a proof of  Theorem 4  based  again on  an analysis of the same tree structure.   For completeness, we present in Section 5 various known results in connexion with Theorem 3.  We show   how the formula (1) is obtained and describe the dynamics of $f_{\lambda,\delta}$ in the case of an irrational rotation number. This section also 
  includes  a  proof that $f_{\lambda,\delta}$ has at most one periodic orbit,  as a consequence of a general statement of unicity.  Finally, we give  some useful properties of Liouville numbers  in an Appendix.

 \section{Transcendental  numbers}

 Let us begin with the following result on transcendence due to Loxton and Van der Poorten  \cite[Theorem 7]{LoVdPA}.
 
\begin{theorem} 
Let $p(x)$ be a non-constant polynomial with algebraic coefficients and $\rho$ an irrational real number, then the power series
$$
\sum_{k=1}^\infty p([k \rho]) \lambda^k
$$
takes a transcendental value for any  algebraic number $\lambda$ with $0<\vert \lambda \vert <1$.
\end{theorem}

We only need Theorem 5 for a polynomial $p(x)$ of degree one, in which case the statement is equivalent to the transcendency of the so-called Hecke-Mahler series 
$$
S_\rho(\lambda)  : =\sum_{k\ge 1} [k \rho] \lambda^k.
$$
A proof can also be found in Sections 2.9 and 2.10 of the monograph \cite{Ni}. See also the survey article \cite{LoVdPB}.

\subsection{Proof of Theorem 1}
We use the fact that  the number $\delta(\lambda,\rho)$ in  $(1)$ may be  expressed in term  of the  Hecke-Mahler  series, thanks to  the useful identity
$$
\sum_{ k \ge 0} \Big([(k+1)\rho]- [k \rho]\Big)\lambda^k = \left({1\over \lambda}-1\right) S_\rho(\lambda), \eqno{(2)}
$$
relating the Hecke-Mahler series with the power series associated to Sturmian sequences.
Using $(2)$, we  rewrite  formula $(1)$ in the form 
$$
\delta(\lambda,\rho)= 1-\lambda +  \frac{(1-\lambda)^2}{\lambda}  \sum_{k=1}^\infty [k\rho] \lambda^{k}=  \frac{(1-\lambda)^2}{\lambda}  \sum_{k=1}^\infty ([k\rho]+1) \lambda^{k}.
\eqno{(3)}
$$
Applying Theorem 5 to the polynomial 
$$
p(z)= \frac{(1-\lambda)^2}{\lambda}z +\frac{(1-\lambda)^2}{\lambda} 
$$
with algebraic non-zero coefficients, we  obtain that $\delta(\lambda,\rho)$ is a transcendental number.
As a consequence, if $\delta$ is an algebraic number, the rotation number $\rho_{\lambda,\delta}$ cannot be an irrational number $\rho$. It is therefore a rational number. 
\cqd

\subsection{Liouville numbers}

We  are concerned here with the diophantine nature of the transcendental number $\delta(\lambda,\rho)$ when $\lambda$ is the reciprocal of an integer. Let us start with the following

\begin{theorem} 
Let  $b\ge 2$ be an integer and let  $0<\rho<1$ be an irrational number. Then,  the number 
$$
S_\rho\left( {1\over b} \right) = \sum_{k=1}^\infty {[k \rho] \over b^k}
$$
is a Liouville number if, and only if, the partial quotients of the continued fraction expansion of the  irrational number  $\rho$ are unbounded.
\end{theorem}
\begin{proof}
By formula $(2)$ and Lemma 11 below, we have to prove equivalently that the Sturmian number
in base $b$
$$
\sum_{ k \ge 1} {[(k+1)\rho]- [k \rho]\over b^k}
$$
is a Liouville number exactly when $\rho$ has unbounded partial quotients.
 P.E. B$\ddot{\o}$hmer has proved  in \cite{Boh} that this number
is a Liouville number when the partial quotients of $\rho$ are unbounded. 
For the converse, notice that the sequence $( [(k+1)\rho] -[k\rho])_{k\ge 1}$ is a Sturmian sequence with values 
in $\{ 0,1\}$ and slope $\rho$. Now, it follows from   \cite[Proposition  11.1]{AdBu} that the associated real number $\displaystyle \sum_{k=1}^\infty  {[(k+1)\rho] -[k\rho] \over b^k}$ in base $b$ is transcendental and it has a finite irrationality exponent exactly when the slope $\rho$ has bounded partial quotients.
\end{proof}

\begin{corollary} 
Let  $\displaystyle \lambda =\frac{1}{b} $, where $b\ge 2$ is an integer,  and let $\rho$ be an irrational number with 
$0< \rho <1$. Then the transcendental number $\displaystyle \delta\left(\frac{1}{b},\rho \right)$ is a Liouville number if and only if the partial quotients of the continued fraction expansion of  $\rho$ are unbounded. 
\end{corollary}
\begin{proof} Using formula (3), we have
$$
\delta \left( \frac{1}{b},\rho \right)=  \frac{b-1}{b} + \frac{(b-1)^2}{b} \sum_{k=1}^\infty {[k\rho] \over b^k}.
$$
It follows from Lemma 11 that $ \delta \left( \frac{1}{b},\rho \right)$ is a Liouville number
if and only if $S_\rho\left({1\over b}\right)$ is a Liouville number.
According to Theorem 6,  this happens exactly  when $\rho$ has unbounded partial quotients. 
\end{proof}

\section{ Rational parameters}

We prove Theorem 2  in this Section. To that purpose, we make use of   results due to  Y. Bugeaud and J.-P. Conze \cite{BuC},  who have shown  that the numbers 
$\displaystyle \frac{1-\lambda}{1-\lambda^q} c\left(\lambda,\frac{p}{q}\right) $ occurring in Theorem 3 (ii) share 
a combinatorial structure parallel to  the classical Stern-Brocot  tree of rational numbers (called the Farey tree in \cite{BuC}).

\subsection{ On the Stern-Brocot tree}

The Stern-Brocot tree, restricted to the interval $[0, 1]$,  may be described as an infinite   sequence of  finite rows, indexed by $k\ge 1$,
consisting of ordered  rational numbers $\displaystyle {p\over q } \in [0,1]$ and  starting with 
$$
\begin{aligned}
k = 1 : \quad & {0\over 1},{ 1\over 1 } 
\\
k= 2 : \quad & { 0\over 1},{1 \over 2},{ 1\over 1}
\\
k= 3 : \quad & { 0\over 1},{ 1\over 3},{1 \over 2}, { 2\over 3}, {1\over 1}
\\
  \dots \, \,   : \quad & 
\dots \dots  \end{aligned}
\eqno{(4)}
$$
The $(k+1)$-th row is constructed from the $k$-th row by inserting  between any pair of  consecutive elements $\displaystyle {p\over q} < { p'\over q'}$ on the $k$-th row,  their  mediant $\displaystyle {p+p'\over q+q'}$.
 Moreover, every rational number $\displaystyle 0 \le {p\over q} \le 1$
appears in (4).

Let $\lambda$ be a real number with $0< \lambda \le 1$. We consider the  analogous tree
$$
\begin{aligned}
k = 1 : \quad & {0\over 1},{ 1\over 1 } 
\\
k= 2 : \quad & { 0\over 1},{1 \over 1+ \lambda},{ 1\over 1}
\\
k= 3 : \quad & { 0\over 1},{ 1\over 1+ \lambda + \lambda^2},{1 \over 1+ \lambda}, { 1+\lambda \over 1+ \lambda + \lambda^2}, {1\over 1}
\\
  \dots \, \,   : \quad & \dots \dots \end{aligned}
\eqno{(5)}
$$
Any element $\displaystyle {P\over Q}$ of the tree  is a rational function in $\lambda$, where the numerator $P=P(\lambda)$ and the denominator $Q= Q(\lambda)$ are polynomials in $\lambda$.
The tree is constructed by a similar  process with the new mediant rule
$$
{P' +\lambda^{q'} P \over Q' + \lambda^{q'}Q}
$$
between consecutive elements $\displaystyle {P\over Q} < {P'\over Q'}$ where  $q' = Q'(1)$.  In the case $\lambda = 1$, we recover
the Stern-Brocot tree (4). Notice that for every $\displaystyle {P\over Q}$ appearing in (5), the rational number $\displaystyle {p\over q} = {P(1)\over Q(1)}$ appears at the same place in (4). It is proved in \cite{BuC} that the rational fractions $\displaystyle {P\over Q}$ appearing in (5) are precisely those which are involved in Theorem 3 (ii);  namely that we have the explicit formula: 
$$
{P\over Q} =\frac{1-\lambda}{1-\lambda^q} c\left(\lambda,\frac{p}{q}\right)
=
{c\left(\lambda,\frac{p}{q}\right) \over 1 + \lambda + \cdots +  \lambda^{q-1}},
$$
where $\displaystyle {p\over q} = {P(1)\over Q(1)}$ and $\displaystyle c\left(\lambda,\frac{p}{q}\right)$ is defined in Theorem 3 (ii). 
We further define $\displaystyle c\left(\lambda,{0\over 1}\right)=0$ and $\displaystyle c\left(\lambda,{1\over 1}\right)=1$   in order to extend the validity of the formula to the endpoints  $0$ and $1$.
Moreover, if $\displaystyle {P\over Q} < {P'\over Q'}$  are adjacent elements on  some row of (5), then 
$$
{P'\over Q'} - {P\over Q} = {\lambda^{q-1} \over Q Q'}
> {\lambda^{q-1}(1-\lambda) \over Q }, \eqno{(6)}
$$
where $q = Q(1)$. It follows from (6) that
$$
{P\over Q} \le  {P+\lambda^{q-1} -\lambda^q \over Q} < {P' \over Q'}.
$$
The assertions (i) and (ii) of Theorem 3 mean   that the rotation number $\rho_{\lambda, \delta}$
is equal to  $p/q$ exactly when $\delta $ is located in the interval 
$$
\left[ {P\over Q} ,  {P+\lambda^{q-1} -\lambda^q \over Q}\right] \subset
 \left[ {P\over Q} ,{P' \over Q'} \right].
 $$

\subsection{A lemma}
From now on, put   $\lambda = a/b$ and $\delta = r/s$ as in Theorem 2.  The following result, based on Liouville's inequality, will be our main tool for proving Theorem 2.

\begin{lemma} 
Let $\displaystyle 0 \le  {p\over q} < {p'\over q'} \le 1$ be consecutive  rational numbers on some row of the Stern-Brocot tree (4). Put
$$
P = c\left(\lambda,\frac{p}{q}\right), \quad Q = 1 + \cdots + \lambda^{q-1}, \quad
P' = c\left(\lambda,\frac{p'}{q'}\right) \quad {\rm and } \quad Q' = 1 + \cdots + \lambda^{q'-1}.
$$
If we assume that 
$$
{P+\lambda^{q-1} -\lambda^q \over Q} < \delta < {P' \over Q'}, \eqno{(7)}
$$
then we have the inequality
$$
b^{ \max (q,q')} \le s b a^{q+q'}.
$$
\end{lemma}
\begin{proof}
It follows  from (6) and (7) that
$$
\begin{aligned}
\max & \left(  {P' \over Q'} -\delta ,  \delta -{P+\lambda^{q-1}-\lambda^q\over Q}\right)
   < {P' \over Q'} - {P+\lambda^{q-1}-\lambda^q\over Q} =  
{\lambda^{q-1}\over Q Q'} -{\lambda^{q-1}-\lambda^q\over Q}
\\
& = {\lambda^{q-1}(1-\lambda)\over Q} \left( {1\over 1-\lambda^{q'}}- 1\right) 
={\lambda^{q+q'-1}\over Q Q'}= {a^{q+q'-1} \over b^{q+q'-1}QQ'},
\end{aligned}
\eqno{(8)}
$$
since $\displaystyle Q' = 1+\cdots + \lambda^{q'-1}= {1-\lambda^{q'}\over 1 -\lambda}$. On the other hand, we have the lower  bound
$$
{P' \over Q'} -\delta= {s P' -rQ'\over sQ'} = {b^{q'-1}(s P' -rQ')\over b^{q'-1}sQ'}  
  \ge  { 1\over b^{q'-1} sQ'},
\eqno{(9)} 
$$
observing that  the numerator
$$
b^{q'-1} (sP'-rQ')=  s \left(   b^{q'-1} + \sum_{j=1}^{q'-2} ([(j+1)p'/q']-[jp'/q'])a^j b^{q'-1-j}\right)
-r \sum_{j=0}^{q'-1} a^j b^{q'-1-j}  
$$
is an integer,  which is necessarily $\ge 1$ since it is the numerator of a positive rational number. 
In a similar way, we find that 
$$
 \delta -{P+\lambda^{q-1}-\lambda^q\over Q} \ge { 1\over b^{q} sQ}.
\eqno{(10)} 
$$
Combining the upper bound (8) with the lower bounds (9) and (10), we obtain the estimates
$$
b^q \le {s a^{q+q'-1}\over Q } \le s a^{q+q'-1}\quad {\rm and} \quad b^{q'-1} \le {sa^{q+q'-1}\over Q' } \le s a^{q+q'-1},
$$
which yield the slightly weaker  inequality $b^{ \max (q,q')} \le s b a^{q+q'}$, since $b>a \ge 1$.
\end{proof}

\subsection{ Proof of Theorem 2} Let $\gamma$ be the golden ratio and set
$$
k = \left[{\gamma  \log (sb) \over \log b -\gamma  \log a}\right]  \ge 1.
\eqno{(11)}
$$
We claim
that there exists a rational $0 \le {p\over q}< 1 $ appearing on the $(k+1)$-st row of (4) such that 
$$ 
 { c\left(\lambda, \displaystyle {p\over q} \right)\over 1 + \ldots + \lambda^{q-1}} \le \delta \le 
{c\left(\lambda, \displaystyle {p\over q}\right)+ \lambda^{q-1} - \lambda^q \over 1 + \ldots + \lambda^{q-1}}.
\eqno{(12)}
$$
 Then, the assertions (i) and (ii) of Theorem 3 yield  that $\rho_{\lambda, \delta} = p/q$.

In order to prove the claim, we first focus to  the $k$-th row of (5). There exist adjacent
elements $\displaystyle {P\over Q} < {P'\over Q'}$ on the $k$-th row of (5) such that $\delta$ is located  in  the  interval
$$
{P\over Q} \le  \delta < {P'\over Q'}.
$$
 Put $\displaystyle {p \over q} ={P(1)\over Q(1)}$ and $\displaystyle {p' \over q'} ={P'(1)\over Q'(1)}$. If $\delta$ falls in the subinterval
$$
{P\over Q} \le  \delta \le {P+ \lambda^{q-1} -\lambda^q\over Q},
$$
 then (12) is established. We may therefore assume that 
$$
{P+ \lambda^{q-1}-\lambda^q \over Q} <   \delta < {P'\over Q'}.
$$
Then, Lemma 8 gives the estimate
$$
b^{ \max (q,q')} \le s b a^{q+q'}. 
\eqno{(13)}
$$

We now consider the mediant $\displaystyle {P''\over Q''} ={P' +\lambda^{q'} P \over Q' + \lambda^{q'}Q}$ of $\displaystyle {P\over Q}$ and $\displaystyle {P'\over Q'}$ which is an element of the $(k+1)$-th row of (5). 
Put  $p'' = P''(1)= p+p'$ and $q'' = Q''(1)= q+q'$.
Since $\displaystyle {p\over q} < {p''\over q''}$ and $\displaystyle {p''\over q''} < {p'\over q'}$ are two pairs of consecutive rational numbers on the $(k+1)$-th row of the Stern-Brocot tree (4), 
we have the inequalities
$$
{P\over Q}< {P+ \lambda^{q-1}-\lambda^q \over Q}< {P''\over Q''} <{P''+ \lambda^{q''-1}-\lambda^{q'' }\over Q''}<  {P'\over Q'}. 
$$
 If $\delta$ belongs to  the subinterval
$$
{P''\over Q''} \le  \delta \le  {P''+ \lambda^{q''-1} -\lambda^{q''}\over Q''},
$$
 then (12) holds true with $p/q$ replaced by $p''/q''$. It thus remains to deal with the intervals
 $$
 {P+ \lambda^{q-1}-\lambda^q \over Q}< \delta <{P''\over Q''}
 \quad {\rm and  } \quad {P''+ \lambda^{q''-1} -\lambda^{q''}\over Q''} < \delta < {P'\over Q'}.
 $$
 In both cases, we shall reach a contradiction, due to the fact that the level $k$ has been selected large enough
 in (11).
 
 Assume first that
 $$
 {P+ \lambda^{q-1}-\lambda^q \over Q}< \delta <{P''\over Q''}.
$$
Then, Lemma 8  applied to the pair of adjacent  rationals $\displaystyle {p\over q} < {p''\over q''}$, gives now the estimate
$$
b^{q+q'} = b^{q''} = b^{ \max (q,q'')} \le s b a^{q+q''}= sba^{2q+q'}. \eqno{(14)}
$$
Raising  the inequalities (13) and (14) respectively to the powers $1/(q+q')$ and $1/(2q+q')$, we immediately obtain the upper bound
$$
b^{\max\left( {q\over q+q'},{q'\over q+q'},{q+q'\over 2 q+q'}\right)} \le (sb)^{ {1\over q+q'}} a.
\eqno{(15)}
$$

Putting  $\displaystyle \xi ={q\over q'}$, we bound from below the exponent
$$
\max\left( {q\over q+q'},{q'\over q+q'},{q+q'\over 2 q+q'}\right) \ge \max\left( {q\over q+q'},{q+q'\over 2 q+q'}\right) 
  = \max \left( {\xi\over 1+\xi},{1+\xi\over 1+2\xi}\right)  \ge {1 \over \gamma}.
$$
For the last inequality, it suffices to note that
$$
{\xi\over 1+\xi} \ge {1 \over \gamma } \,\,   \iff \,\,  \xi \ge {1 \over \gamma -1}
\quad \mbox{and} \quad
{1+ \xi\over 1+2\xi} \ge {1 \over \gamma } \,\,  \iff \,\, { \gamma -1 \over 2-\gamma}= {1 \over \gamma -1} \ge \xi.
$$
Observe that for each pair of adjacent numbers $\displaystyle {p\over q}< {p'\over q'}$ on the $k$-th row of (4), we have the lower bound $q+q' \ge k+1$. We then deduce  from (15) that 
$$ 
b^{1\over \gamma} \le (sb)^{{1\over k+1}} a, 
$$ 
or equivalently that $\displaystyle  k+1  \le {\gamma  \log (sb) \over  \log b -\gamma  \log a}$, in contradiction with our choice (11) for $k$. 

The second interval
$$
{P''+ \lambda^{q''-1} -\lambda^{q''}\over Q''} < \delta < {P'\over Q'}
$$
is treated in a similar way. In this case, the analogue of (14) writes now 
$$
b^{q+q'} = b^{q''} = b^{ \max (q',q'')} \le s b a^{q'+q''}= sba^{q+2q'}.
$$
Permuting  $q$ and $q'$, we are led to the same computations. The claim is proved.

It is well-known that each rational number on the $(k+1)$-th row of the Stern-Brocot tree (4) has a denominator $q$ bounded by the Fibonacci number
$$
F_{k+2} =  {1\over \sqrt{5}} \left( \gamma^{k+2} - \left(-{1/\gamma}\right)^{k+2} \right)
\le 
{1\over \sqrt{5}}  \gamma^{2 + {\gamma  \log (sb) \over  \log b -\gamma \log a}} +{1\over \sqrt{5}}
\le 
   \gamma^{2+  {\gamma \log (sb) \over  \log b -\gamma \log a}} ,
$$
since $\displaystyle k \le{\gamma \log (sb) \over  \log b -\gamma \log a}$ by (11). 
Theorem 2 is proved.
\cqd

\subsection{A remark of optimality}
We have considered in the preceding argumentation two consecutive rows of index $k$ and $k+1$ in the tree $(5)$.
Using the subsequent rows of indices $k+2, \dots$ does not necessarily produce a substantially better result, as  the following construction shows. 

Let 
$$
{\sqrt{5}-1\over 2}= {1\over \gamma} = [0; 1,1, 1, \dots]
$$
be the continued fraction expansion of the irrational number ${\sqrt{5}-1\over 2}$ whose sequence of convergents is $ 0 ={F_0\over F_1}, 1= {F_1\over F_2}, {F_2\over F_3}, \dots$, where we recall that $(F_l)_{l\ge 0}$ denotes the Fibonacci sequence. For every $l\ge 0$, set 
$$
P_l = c\left(\lambda, {F_l\over F_{l+1}}\right)
\quad\hbox{\rm and}\quad
Q_l = 1 + \cdots + \lambda^{F_{l+1}-1},
$$
so that ${P_l\over Q_l}$ is the element in the tree $(5)$ associated to the convergent ${F_l\over F_{l+1}}$. Now, 
$$
{F_{l+2}\over F_{l+3}}= {F_{l+1}+F_{l}\over F_{l+2}+ F_{l+1}}
$$ 
is the mediant of
${F_{l+1}\over F_{l+2}}$ and ${F_l\over F_{l+1}}$ in the Stern-Brocot tree $(4)$. Accordingly, ${P_{l+2}\over Q_{l+2}}$ is the mediant of ${P_{l+1}\over Q_{l+1}}$ and ${P_l\over Q_{l}}$ in the  tree $(5)$. It follows that
 ${P_l\over Q_l }< {P_{l+1}\over Q_{l+1}}$ (resp.  ${P_{l+1}\over Q_{l+1}} < {P_{l}\over Q_{l}}$)  are adjacent elements on the 
$(l+1)$-th row of the tree $(5)$ when $l$ is even (resp. odd). Therefore,  the open intervals
$$
I_l:=
\begin{cases}
\left( { P_l + \lambda^{F_{l+1}-1} -\lambda^{F_{l+1}}\over Q_l} , {P_{l+1}\over Q_{l+1}}\right)
 \quad \hbox{\rm if} \quad  l \,\, \hbox{\rm is  even}
\\
\left( { P_{l+1} + \lambda^{F_{l+2}-1} -\lambda^{F_{l+2}}\over Q_{l+1}} , {P_l\over Q_l}\right)
\quad \hbox{\rm if} \quad l \,\, \hbox{\rm  is  odd}
\end{cases}
,
\quad l\ge 0,
$$
form a decreasing sequence of nested intervals whose intersection reduces to the transcendental number $\delta(\lambda, 1/\gamma)$. Now, assuming that the rational $\delta={r\over s}$ belongs to the interval $I_l$, Lemma 8 gives us the inequality
$$
b \le (sb)^{{1\over F_{l+2}} }a^{{F_{l+1}+ F_{l+2}\over F_{l+2}}}.
$$
Observe that the exponent  ${F_{l+2}+ F_{l+1}\over F_{l+2}}$ converges to $\gamma$ as $l$ tends to infinity.
Thus, the  constraint $b> a^\gamma$ assumed in Theorem 2 cannot be essentially relaxed by using Lemma 8. However, we have no reason to believe that Lemma 8 itself is optimal.

\section{ Proof of Theorem 4}

Let $0 < \lambda <1$. We  prove that the set 
$$
\E = \{ \delta \in I : \rho_{\lambda,\delta} \mbox{ is irrational} \}.
$$
has null Hausdorff dimension,  using  properties of the tree (5) displayed in Section 3. It turns out that $\E$ has a natural structure of Cantor set induced by the tree (5). 

For any $k\ge 1$, both  trees (4) and (5) have $2^{k-1}+1$ elements on their  $k$-th row. Let us number
$$
\begin{aligned}
& 0= {p_{0,k}\over q_{0,k}} < {p_{1,k}\over q_{1,k}}< \dots < {p_{2^{k-1},k}\over q_{2^{k-1},k}}=1,
\\
& 0= {P_{0,k}\over Q_{0,k}} < {P_{1,k}\over Q_{1,k}}< \dots < {P_{2^{k-1},k}\over Q_{2^{k-1},k}}=1,
\end{aligned}
$$
their respective elements by increasing order, and put
$$
\begin{aligned}
& 
I_{j,k} = \left( {P_{j,k} + \lambda^{q_{j,k}-1} -\lambda^{q_{j,k}}\over Q_{j,k}}, {P_{j+1,k} \over Q_{j+1,k}}\right) \subset I , \quad 0 \le j \le 2^{k-1}-1 , k\ge 1,
\\ & 
\E_k = \bigcup_{j =0}^{2^{k-1}-1} I_{j,k} \subset I , \quad k\ge 1.
\end{aligned}
$$
Each interval $I_{j,k}$ in $\E_k$ contains  two intervals from $\E_{k+1}$ : 
$$
I_{j,k} \cap \E_{k+1} = I_{j',k+1} \cup I_{j'+1,k+1}, 
$$
where $j'$ and $j'+1$ are the respective indices  of ${p_{j,k}\over q_{j,k}}$ and  of the mediant ${p_{j,k} + p_{j+1,k}\over q_{j,k} + q_{j+1, k}}$ on the $(k+1)$-th row of the Stern-Brocot tree (4). 
Hence, for every $k\ge 1$, we have the inclusion
$$
\E_{k+1} \subset \E_k.
$$

It follows from Theorem 3 that
$$
\E =  \bigcap_{k\ge 1}\E_k= \bigcap_{k\ge 1}\bigcup _{j=0}^{2^{k-1}-1} I_{j,k} .
$$

\begin{lemma}
The upper bounds
$$
\max_{0\le j \le 2^{k-1}-1} | I_{j,k} | \le \lambda^{k}  \quad \mbox{and} \quad
\sum_{j=0}^{2^{k-1}-1} | I_{j,k} |^\sigma \le \sum_{n\ge k } n \lambda^{\sigma n},
$$
hold 
for any integer $k\ge 1$ and any real number $\sigma >0$,
where the bars $| \cdot |$ indicate here the length of an interval.
\end{lemma}

\begin{proof}
It follows from (6) and (8)  that
$$
| I_{j,k} | = {\lambda^{q_{j,k} + q_{j+1,k}-1} \over Q_{j,k}Q_{j+1,k}} \le \lambda^{q_{j,k} + q_{j+1,k}-1},
\quad k\ge 1 , 0\le j \le 2^{k-1}-1.
$$
As $ {p_{j,k}\over q_{j,k}}<  {p_{j+1,k} \over q_{j+1, k}}$ are consecutive elements on the $k$-th row of the  Stern-Brocot tree, the lower bound 
$
q_{j,k} + q_{j+1,k} \ge k+1
$
holds  for $0\le j \le 2^{k-1}-1$. The first upper bound of Lemma 9 immediately follows. For the second one, observe that 
$$
 \sum_{j=0}^{2^{k-1}-1} | I_{j,k} |^\sigma \le 
 \sum_{n\ge k+1} \mbox{Card} \left\{ j: q_{j,k} + q_{j+1,k} = n \right\} \lambda^{\sigma (n-1)}
 \le 
 \sum_{n\ge k+1 } (n-1) \lambda^{\sigma (n-1)} , 
 $$
 the last inequality coming  from the fact that $q_{j,k}+q_{j+1,k}$ is the denominator of  
 the mediant ${p_{j,k} +p_{j+1,k}\over q_{j,k}+q_{j+1,k}}$ located on the  $(k+1)$-th row and that for any $n\ge 2$ a given row of the Stern-Brocot tree contains at most $n-1$ elements having denominator $n$. 
 \end{proof}
 
Observe that the series $ \sum_{n\ge 1 } n \lambda^{\sigma n}$ converges for any $\sigma>0$.
 Then, Theorem 4 follows from Lemma 9 by using standard arguments on the $\sigma$-dimensional Hausdorff measure for $0 <\sigma \le 1$. See for instance Corollary 2 in \cite{BeTa}.

 \section{On the rotation number $\rho_{\lambda,\delta}$}
 
 We have collected in this section further information concerning the basic Theorem 3 and the dynamics of the transformation $f_{\lambda,\delta}$. 
 
 Let us begin with the definition of the rotation number.
  In order to compute the rotation number of $f$, it is convenient to introduce a function $F:  \R \to  \R$ called 
 {\it suspension or lift of $f$ on $\R$}. Our lift $F$ will be defined by 
$$
F=F_{\lambda,\delta}:x \in \R \mapsto \lambda\{ x\}+\delta+[x],
$$
where $[x]$ is the integer part of the real number $x$. Its name comes from the fact that $F$ satisfies the following properties:
\\
(i) Let $\pi:x\in\R \mapsto \{ x\} $ be the canonical projection of $\R$ on $I$, then, for every $x\in \R$, 
$$
\pi(F(x))=f(\pi(x)).
$$
(ii) $F(x+1)=F(x)+1$, for every $x\in \R$.\\
(iii) $F$ is an increasing  function on $\R$ which is continuous on each interval of $\R \setminus \Z$ and right continuous on $\Z$.

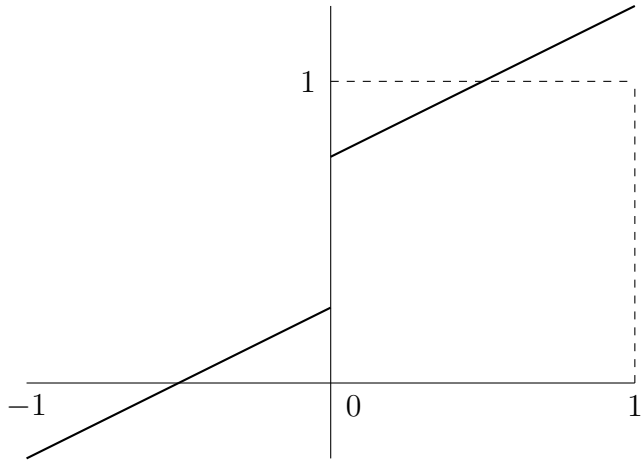
\begin{figure}[ht]
\begin{tikzpicture}
\draw (-4,0) -- (4,0);
\draw (0,-1) -- (0,5);
\draw[dashed] (4,0) -- (4,4);
\draw[dashed] (0,4) -- (4,4);
\draw[thick] (0,3) -- (4,5);
\draw[thick] (-4,-1) -- (0,1);
\node   at (0.3 ,-0.3) {$0$};
\node   at (-0.3,4) {$1$};
\node   at (-4,-0.3){$-1$};
\node   at (4,-0.3) {$1$};
\end{tikzpicture}
\caption{ Plot of $F_{1/2, 3/4}(x) $ in the interval $-1\le x < 1$ }
\end{figure}

Let $x_0 \in \R$. It is known (see \cite{RT}) that the limit 
$$
\rho_{\lambda,\delta} = \lim_{k \to \infty} \frac{F_{\lambda,\delta}^k(x_0)}{k} 
$$ 
exists and does not depend on the initial point $x_0$.  The number $\rho_{\lambda,\delta}$ is called   {\it rotation number} of the map $f=f_{\lambda,\delta}$ and $0\le \rho_{\lambda,\delta}<1$.

\subsection{The irrational rotation number case}

Let $0<\lambda<1$ be fixed. As indicated in Theorem 3, it is known that the graph of the application  $ \delta  \mapsto \rho_{\lambda,\delta} $ is a { \it Devil's staircase}, meaning   that this map is a continuous non-decreasing function  sending  $[0,1)$ onto 
$ [0,1)$ which takes the rational value $p/q$ on the whole interval 
$$
\frac{1-\lambda}{1-\lambda^q} c\left(\lambda,\frac{p}{q}\right) \le \delta \le \frac{1-\lambda}{1-\lambda^q} \left( c\left(\lambda,\frac{p}{q}\right) + \lambda^{q-1}-\lambda^q \right).
$$
 Notice that if we select a sequence of reduced rationals  $p/q$ tending to an irrational number $\rho$, so that $q$ tends to infinity,
then the above intervals shrink to the point
$$
\delta(\lambda,\rho) =(1-\lambda) \left( 1 +  \sum_{k=1}^{+\infty }\left(  \left[ (k +1)\rho \right] -  \left[ k \rho \right] \right)\lambda^{k} \right),
$$
 since the integer part function $x \mapsto [x]$ is continuous at non-integer  $x$. Therefore, for each  irrational number $\rho$ with $0< \rho <1$, there is exactly one value of $\delta$ for which $\rho = \rho_{\lambda,\delta} $ and this value is given by the series $\delta(\lambda,\rho)$.

Fix now an irrational number $\rho$ with $0 < \rho <1$.
Following the approach given in \cite{Co}, define a function $\varphi : \R \to \R$ by the formula
$$
\varphi(t) =  \sum_{k\ge 0} \lambda^k\Big(\delta(\lambda,\rho) + (1-\lambda)[t-(k+1) \rho] \Big).
$$
It has been established in \cite{Co} that the function  $\varphi : \R \to \R$ is an increasing function which has the following properties:

$(C1)$ $\varphi(t+\rho)  = F_{\lambda,\delta(\lambda,\rho)}(\varphi(t)) ,  \quad \forall t  \in \R$, 

$(C2)$ $\varphi  (t+1)= \varphi  (t)+1$, and 

$(C3)$ $\varphi  (0)=0$.

Reducing modulo $1$, this establishes that $ f_{\lambda,\delta(\lambda,\rho)}$ and the circle rotation map $R_\rho : x\in I \mapsto x+ \rho \mod 1$ are topologically conjugated on the limit set  $C = \varphi(I) \subset I$. 

Notice that $(C2)$ and $(C3)$ imply that $[\varphi(t) ] = [t]$ for every $t\in \R$,  from which  $(C1)$
easily follows.
The property $(C3)$ is equivalent to the formula  (3) for $\delta(\lambda,\rho)$ since
$$
\begin{aligned}
\varphi (0)& =  \sum_{k\ge 0} \lambda^k\Big(\delta(\lambda,\rho) + (1-\lambda)[-(k+1) \rho] \Big)
= {\delta(\lambda,\rho) \over 1-\lambda} +(1-\lambda) \sum_{k\ge 0} \lambda^k[-(k+1) \rho]
\\
&= 
{\delta(\lambda,\rho) \over 1-\lambda} -(1-\lambda) \sum_{k\ge 0} \lambda^k(1 + [(k+1) \rho])
={\delta(\lambda,\rho) \over 1-\lambda} -1- {1-\lambda\over \lambda}  \sum_{k\ge 1} \lambda^k [k \rho].
\end{aligned}
$$

We now link the preceding facts  to the dynamics of $f_{\lambda,\delta(\lambda,\rho)}$, or equivalently to the iterates of the real function $F_{\lambda,\delta(\lambda,\rho)}$ by reducing modulo $1$. 
For any  $x_0 \in \R$, we define the orbit of  $x_0$ under $F_{\lambda,\delta(\lambda,\rho)}$ iteratively  by
$$
x_{k+1}=F_{\lambda,\delta(\lambda,\rho)}(x_k)=\lambda x_k+\delta(\lambda,\rho)+(1-\lambda)[x_k], \  \mbox{ for every} \ k\ge 0. 
$$  
 We then easily obtain by induction on $n$ the formula
$$
x_n= \lambda^n x_0 + \sum_{k=0}^{n-1} \lambda^k\Big(\delta(\lambda,\rho) + (1-\lambda)[x_{n-k-1}]\Big),\  \mbox{ for every} \ n\ge 0.
$$

It is proved in \cite{Co} that for each initial point $x_0$, there exists $t_0\in \R$ such that
$$
[x_k ]=[t_0+ \rho k] \ \mbox{ for every } \ k\geq 0.
$$
Therefore
$$
x_n= \lambda^n x_0 + \sum_{k=0}^{n-1} \lambda^k\Big(\delta(\lambda,\rho) + (1-\lambda)[t_0+ n\rho -(k+1)\rho]\Big)
= \varphi(t_0 + n\rho) + {\mathcal  O} (\lambda^n),
$$
as $n$ tends to infinity. Hence the $f_{\lambda,\delta(\lambda,\rho)}$-orbit of $x_0$ modulo  1 converges 
exponentially fast to the orbit of $\varphi(t_0)$ modulo 1 on $C$.

\subsection{The rational rotation number case}
According to \cite{GT}, any  contracted rotation $f_{\lambda,\delta}$  with rational rotation number $\rho_{\lambda,\delta}$ has a unique periodic orbit in an extended meaning, allowing multivalues at the discontinuity point $0\in \R/ \Z$. Then, for every $x\in I$, the $\omega$-limit  set
$$
 \omega(x) : = \cap_{\ell \ge 0} \overline{\cup_{k \ge \ell} f_{\lambda,\delta}^k(x)}
$$
 coincides with this periodic orbit.

 Let $f : I \to I$ be a  map of the unit interval which is continuous outside finitely many points.  We say that $f$ is a {\it piecewise contracting map}, when  there exists $0<\lambda<1$ such that for every open  interval  $J \subset I$ on which $f$ is continuous, the inequality $| f(x)-f(y)|  \le \lambda | x-y|  $ holds  for any $x,y \in J $ (see \cite{NP}).
 We view as well $f$ as a circle map $f : \R/\Z \to \R/\Z$ by identifying the sets $\R/\Z$ and  $ I$,  thanks to the canonical  bijection   $ I \hookrightarrow \R \to \R/\Z$. 
We now  prove  that $f=f_{\lambda,\delta}$  has at most one periodic orbit, as a corollary of the  more general statement

\begin{theorem} 
Let $f$ be a piecewise contracting and orientation-preserving circle map.
Assume that $f$ has  a unique  point of  discontinuity on the circle $\R/\Z$ and  that $f$ is right continuous at this point.
Then $f$ has at most one periodic orbit. 
\end{theorem}

\begin{proof} 
We may assume without loss of generality that the discontinuity is the origin $0\in \R /\Z$. This is the case for any  contracted rotation $f_{\lambda,\delta}$ (see Figure 2).
 Then, there exists a lift $F$ of $f$ satisfying the three  properties (i), (ii), (iii) appearing on page 11,   which are verified  by the function $F_{\lambda,\rho}$. Moreover, there exists $0<\lambda<1$ such that for each integer $n$, we have 
inequalities of the form
$$
0 \le F(y) -F(x) \le \lambda (y-x)
\quad\hbox{\rm for every } \quad
n \le x \le y < n+1.
\eqno{(16)}
$$
Let $O_1$ and $O_2$ be two periodic orbits of $f$. 
For $i=1, 2$, define $x_i$ as the smallest element of $O_i$ in $I$. Exchanging possibly $O_1$ and $O_2$, we may assume that  $0 \le x_1\le x_2<1$. 
We then claim that 
$$
[F^k(x_1)]= [F^k(x_2)] 
\quad\hbox{\rm and} \quad
0 \le F^k(x_2)-F^k(x_1)\le \lambda^k(x_2-x_1)
\eqno{(17)}
$$
 for every integer $k \ge 0 $. The proof is performed by induction  on $k$ and the assertion clearly holds for $k=0$. We have
 $$
 0 \le F^{k+1}(x_2)-F^{k+1}(x_1) = F(F^k(x_2))-F(F^k(x_1))\le \lambda(F^k(x_2)-F^k(x_1))\le\lambda^{k+1}(x_2-x_1),
 $$
 by (16) and (17).   Assume on the  contrary  that 
$$
[ F^{k+1}(x_2) ] \ge [ F^{k+1}(x_1)] +1,
$$
however  $[ F^{k}(x_2) ] = [ F^{k}(x_1)]$. We thus have 
$$
F^{k+1}(x_1) < [F^{k+1}(x_1)]+ 1 \le [F^{k+1}(x_2)]  \quad  \mbox{and}  \quad F^{k+1}(x_2)- F^{k+1}(x_1) \le  \lambda^{k+1}(x_2-x_1),
$$
so that 
$$
 f^{k+1}(x_2)= \{F^{k+1}(x_2)\}=F^{k+1}(x_2)-[F^{k+1}(x_2)] <  F^{k+1}(x_2) - F^{k+1}(x_1) \le\lambda^{k+1}(x_2-x_1) \le x_2
$$
which contradicts the choice of $x_2$. The claim is proved.

Reducing modulo 1 and letting $k$ tend to infinity, we deduce from (17) by periodicity that $ f^k(x_2) = f^k(x_1)$ for any  $k\ge 0$. We have proved that 
$O_1=O_2$.

\end{proof}

\section*{Appendix : Liouville numbers}

A real number $\xi$ is  a {\it Liouville number}, if for every integer number $k\ge 1$, there exist integers numbers $p_k$ and $q_k$ with $q_k \ge 2$ such that
$$
0< \left\vert \xi - {p_k \over q_k} \right\vert < {1 \over q_k^k}.
$$
Liouville's  inequality shows that $\xi$ is a transcendental number.
 The sequence $(q_k)_{k\ge 1}$ is thus unbounded.
For completeness, we include the proof of the following lemma:

\begin{lemma} 
Let $\xi$ be  Liouville and $r$ and $s$ be nonzero rational numbers, then $r\xi+s$ is  a Liouville number.
\end{lemma}
{\it Proof.} 
We write $\displaystyle r= { a \over b}$ and $\displaystyle s= { c \over d}$, where $a,b,c,d \in \Z$ with $b\ge 1$ and $d\ge 1$.
We claim that $\displaystyle { a \over b} \xi+ { c \over d}$ is  a Liouville number.
Let $m\ge 2$ be arbitrarily fixed. 
There exists $k=k(m)$ such that $k>m$ and $q_k\ge \vert a \vert b^{m-1}d^m$. We have
$$
 \left\vert {a \over b}  \xi + { c \over d} - { adp_k+bcq_k \over bdq_k} \right\vert 
= { \vert a \vert \over b} \left\vert \xi -{p_k \over q_k} \right\vert  
 <  { \vert a \vert \over b} {1 \over q_k^k}\le { \vert a \vert \over b} {1 \over \vert a \vert b^{m-1}d^m q_k^{k-1}} \le {1 \over (bdq_k)^{m}}
 $$
 which proves that $r\xi+s$ is a Liouville number.
 \cqd

{\bf Acknowledgements.} 
Our work originates from  a question of Mark Pollicott about the rationality of the rotation number for rational values
of the parameters. We are indebted to him for valuable discussions on the subject.
We would like to thank the anonymous referee for a careful reading of the  manuscript whose remarks enhanced greatly our text. We graciously acknowledge the support of R\'egion Provence-Alpes-C\^ote d'Azur through the project APEX {\it Syst\`emes dynamiques: Probabilit\'es et Approximation Diophantienne} PAD and CEFIPRA through project Proposal No. 5801-1.

\vskip 5mm

\centerline{\scshape {\rm Michel} Laurent {\rm and } {\rm Arnaldo} Nogueira(*) \footnote{(*) Partially
 supported by the programs MATHAMSUD projet No. 36465XD PHYSECO and MATHAMSUD projet No. 38889TM DCS: Dynamics of Cantor systems.}}

{\footnotesize
 \centerline{Aix Marseille Univ, CNRS, Centrale Marseille, I2M, Marseille, France}
 \centerline{michel-julien.laurent@univ-amu.fr and  arnaldo.nogueira@univ-amu.fr}}

\end{document}